\documentclass{amsart}
\usepackage{amssymb}
\usepackage{enumerate}
\usepackage{bbm}
\usepackage[final]{hyperref}
\usepackage{microtype}

\newtheorem{theorem}{Theorem}
\newtheorem{corollary}[theorem]{Corollary}
\newtheorem{proposition}[theorem]{Proposition}
\newtheorem{lemma}{Lemma}

\theoremstyle{definition}
\newtheorem{definition}{Definition}

\theoremstyle{remark}

\newcommand{\eps}		{\varepsilon}
\newcommand{\R}		{\mathbbm{R}}

\newcommand{\A}		{\mathcal{A}}
\newcommand{\B}		{\mathcal{B}}
\newcommand{\G}		{\mathcal{G}}
\newcommand{\U}		{\mathcal{U}}
\renewcommand{\S}		{\mathcal{S}}
\renewcommand{\P}		{\mathfrak{A}}

\newcommand{\norm}[1]	{\lVert#1\rVert}

\renewcommand{\exp}[1]	{\operatorname{exp}\left( #1 \right)}

\renewcommand{\Pr}[2][]	{\mathbbm{P}_{#1}\left( #2 \right)}
\newcommand{\Ex}[1]	{\mathbbm{E}\left( #1 \right)}
\newcommand{\E}		{\mathbbm{E}}
\newcommand{\ind}[1]	{\mathbbm{1}_{#1}}

\newcommand{\lam}		{\lambda_{\ast}}
\newcommand{\K}		{\kappa}
\newcommand{\notcx}	{\nleftrightarrow}

\DeclareMathOperator{\Bin}{binomial}
\DeclareMathOperator{\essinf}{ess\,inf}
\DeclareMathOperator{\diam}{diam}

\begin{document}

\title{Connectivity of Inhomogeneous Random Graphs}

\author{Luc Devroye}
\address{School of Computer Science\\ McGill University\\ Montreal, Canada H3A 2K6.}
\email{luc@cs.mcgill.ca}
\author{Nicolas Fraiman}
\address{Department of Mathematics and Statistics\\ McGill University\\ Montreal, Canada H3A 2K6.}
\email{fraiman@math.mcgill.ca}

\thanks{The research of both authors was sponsored by NSERC Grant A3456}

\date{\today}

\maketitle

\begin{abstract}
We find conditions for the connectivity of inhomogeneous random graphs with intermediate density. Our results generalize the classical result for $G(n,p)$, when $p = c\log n/n$. We draw $n$ independent points $X_i$ from a general distribution on a separable metric space, and let their indices form the vertex set of a graph. An edge $(i,j)$ is added with probability $\min (1, \K(X_i, X_j) \log n / n)$, where $\K \geq 0$ is a fixed kernel. We show that, under reasonably weak assumptions, the connectivity threshold of the model can be determined.
\end{abstract}

%%%%%%%%%%%%%%%%%%%%%%%%%%%%%%%%%%%%%%%%%%%%%%%%%%%%
\section{Introduction}\label{sec:intro}

We study the connectivity of inhomogeneous random graphs, where edges are present independently but with unequal edge occupation probabilities. A discrete version of the model was introduced by S{\"o}derberg \cite{Sod02}.
The sparse case (when the number of edges is linear in  the number $n$ of vertices) was studied in substantial detail in the seminal paper by Bollob\'as, Janson and Riordan \cite{BJR07}, where various results have been proved, including the critical value for the emergence of a giant component, and bounds on the connected component sizes in the super and subcritical regimes.
The dense case (when the number of edges is quadratic in $n$) has developed into a deep and beautiful theory of graph limits started by Lov\'asz and Szegedy \cite{LS06} and further studied in depth by Borgs, Chayes, Lov\'asz, S\'os and Vesztergombi \cite{BCLSV08,BCLSV12} and by Bollob\'as, Borgs, Chayes and Riordan \cite{BBCR10} among others.

Models with intermediate density (a number of edges that is more than linear but less than quadratic in $n$) can be obtained by defining the edge probabilities with a different scaling. Although there are connections to the other cases they lead to very different properties. The intermediate density case has not received much attention but it is of particular interest since it is the natural setting to study the transition for connectivity and other related properties.

\subsection{The model} 
In this paper we follow the notation from \cite{BJR07} with some minor changes. We also use the following standard notation: we write $(\:\cdot\:)_+$ for the positive part, $f =O(g)$ if $f/g$ is bounded and $f = o(g)$ if $f/g\to 0$. We say that a sequence of events holds \emph{with high probability}, if it holds with probability tending to $1$ as $n\to\infty$.

Let $\S$ be a separable metric space and $\mu$ a Borel probability measure on $\S$. Let $X_1,\dots,X_n$ be $\mu$-distributed independent random variables on $\S$. In what follows, $X$ denotes another variable independent of $X_1,\dots,X_n$ with the same distribution. Let $\K:\S\times\S \to \R^+$ a non-negative symmetric integrable kernel, $\K \geq 0$ and $\K\in L^1(\S\times\S, \mu\otimes\mu)$.

\begin{definition}
The \emph{(intermediate) inhomogeneous random graph} with kernel $\K$ is the random graph $G(n,\K) = (V_n,E_n)$ where the vertex set is $V_n = \{1,\dots, n\}$ and we connect each pair of vertices $i,j\in V_n$ independently with probability $p_{ij} = \min\{1, \K(X_i,X_j)p_n\}$ where $p_n = \log n / n$.
\end{definition}

\begin{definition}
Let
\[
\lambda(x) = \int_\S \K(x,y) d\mu(y) \qquad\text{and}\qquad
\lambda_2(x) = \left(\int_\S \K(x,y)^2 d\mu(y)\right)^{1/2}.
\]
We call $\lam = \essinf \lambda(x)$ the \emph{isolation parameter}.
\end{definition}

\begin{definition}
A kernel $\K$ on $(\S,\mu)$ is \emph{reducible} if there exists a set $\A \subset \S$ with $0 < \mu(\A) < 1$ such that $\K = 0$ almost everywhere on $\A \times \A^c$. Otherwise $\K$ is \emph{irreducible}.
\end{definition}
If $\K$ is reducible then we cannot expect the whole graph $G(n,\K)$ to be connected since almost surely there are no edges between the sets $A = \{i: X_i\in \A\}$ and $A^c$. Hence, we shall restrict our attention to the irreducible case.

\subsection{Results} 
The main result we prove is a generalization of the classical result of Erd\H{o}s and Renyi \cite{ER59},\cite{ER60} for $G(n,p)$ stated below.

\begin{theorem}\label{thm:main}
If $\K$ is irreducible, continuous $(\mu\otimes\mu)$-almost everywhere and $\lambda_2 \in L^\infty(\S,\mu)$ then
\[
\lim_{n\to \infty} \Pr{G(n,\K) \text{ is connected}} = \begin{cases}
0 & \text{if }\; \lam < 1, \\
1 & \text{if }\; \lam > 1.
\end{cases}
\]
\end{theorem}
Note that changing the kernel in a set of $\mu\otimes\mu$ measure zero defines the same graph $G(n,\K)$ almost surely. Therefore what we actually need is that there is a version of $\K$ (i.e., $\tilde{\K}$ such that $\tilde{\K}=\K$ almost everywhere) that is continuous almost everywhere.

The theorem is proved in two parts. In Section \ref{sec:isolated} we prove that when $\lam < 1$ the graph $G(n,\K)$ is disconnected with high probability. We prove this under milder conditions for the kernel $\K$ using the second moment method. In Section \ref{sec:connectivity} we prove that when $\lam > 1$ we have connectivity with high probability. To prove this we start by showing that every component should be at least of linear size using concentration inequalities. Then we use a discretization argument to prove that any two such components must meet.

If $G$ is a group acting transitively on $\S$ with invariant measure $\mu$ and $\K$ is an invariant kernel, we say we are in the \emph{homogeneous} case. We can specialize Theorem \ref{thm:main} for this case. Since there exists $g\in G$ such that $g x = z$ then we have $\lambda(x) = \int_\S \K(g x, g y) d\mu(y) = \int_\S \K(z,w) d\mu(w) = \lambda(z) = \lam$ thus $\lambda(x)$ and $\lambda_2(x)$ are independent of $x\in \S$.
Here $\K\in L^2(\S\times\S,\mu\otimes\mu)$ is enough to guarantee that $\lambda_2 \in L^\infty(\S,\mu)$. Therefore we have the following

\begin{corollary}
If $\K\in L^2(\S\times\S,\mu\otimes\mu)$ is homogeneous, irreducible and continuous $(\mu\otimes\mu)$-almost everywhere then
\[
\lim_{n\to \infty} \Pr{G(n,\K) \text{ is connected}} = \begin{cases}
0 & \text{if }\; \lam < 1, \\
1 & \text{if }\; \lam > 1.
\end{cases}
\]
\end{corollary}

The Erd\H{o}s-Renyi random graph and the random bipartite graph are both particular cases in which $\S$ has only one or two points respectively. Another example is given by taking $\S = [0,1)$ with Lebesgue measure $\mu$, and $\K(x, y) = h(x - y)$ for a periodic even function. In general, we can take $\K(x,y) = f(d(x,y))$ where $d$ is an invariant metric with corresponding Haar measure $\mu$. However, the random geometric graph introduced by Gilbert \cite{Gil61} whose connectivity threshold was determined by Penrose \cite{Pen97} (and other properties were studied in depth in the monograph \cite{Pen03}) is not included in this Corollary because it cannot be represented with a fixed $\K$ in $L^2$.

%%%%%%%%%%%%%%%%%%%%%%%%%%%%%%%%%%%%%%%%%%%%%%%%%%%%
\section{Occurrence of isolated vertices}\label{sec:isolated}

In this Section we prove that the graph is disconnected with high probability when $\lam < 1$. We prove it by showing that in this case with high probability isolated vertices are going to exist on the graph. The technique is based on the second moment method.

\begin{theorem}
If $\lambda_2 \in L^2(\S,\mu)$ and $\lam < 1$ then $G(n,\K)$ is disconnected with high probability.
\end{theorem}
\begin{proof}
Let $N$ be the number of isolated vertices. We can write $N = \sum_{i=1}^n I_i$ where $I_i$ is the indicator that vertex $i$ is isolated. Since $\lam < 1$ there exists $\eps > 0$ such that the set $\B = \{x \in \S: \lambda(x) < 1-\eps \}$ has measure $\mu(\B) > 0$. We are focusing only on the points that lie in $\B$. Define $N_\B = \sum_{i=1}^n Y_i$ where $Y_i$ is the indicator that vertex $i$ is isolated and $X_i \in \B$. Clearly $N \geq N_\B$. We show that $\lim_{n\to\infty} \Pr{N_\B > 0} = 1$ using the second moment method.
By the Cauchy--Schwarz inequality we have that
\[
\Pr{N_\B > 0} \geq \frac{\Ex{N_\B}^2}{\Ex{N_\B^2}}.
\]
Since $\Ex{N_\B} = n \Ex{Y_1}$ and $\Ex{N_\B^2} = \Ex{N_\B} + n(n-1)\Ex{Y_1Y_2}$, we are done if
\[
\lim_{n\to\infty} n\Ex{Y_1} = \infty \qquad\text{and}\qquad
\limsup_{n\to\infty} \frac{\Ex{Y_1Y_2}}{\Ex{Y_1}\Ex{Y_2}} \leq 1.
\]
For the first limit consider
\begin{align}
\Ex{Y_1} &= \Ex{\ind{[X_1\in \B]}\prod_{j=2}^n \ind{[(1,j)\notin E_n]}} \nonumber \\
&= \int_\B \prod_{j=2}^n \Ex{\big(1-\K(X_j,x)p_n\big)_+} d\mu(x) \nonumber \\
&= \int_\B \Ex{\big(1-\K(X,x)p_n\big)_+}^{n-1} d\mu(x) \nonumber \\
&\geq \int_\B (1-\lambda(x)p_n)^{n-1} d\mu(x) \label{eq:expected} \\
&\geq (1-(1-\eps)p_n)^{n-1}\mu(\B). \nonumber
\end{align}
Therefore,
\begin{align*}
\lim_{n\to\infty} n\Ex{Y_1}
&\geq \lim_{n\to\infty} n(1-(1-\eps)p_n)^{n-1}\mu(\B)  \\
&= \lim_{n\to\infty} n e^{-(1-\eps)np_n}\mu(\B) \\
&= \lim_{n\to\infty} n^{\eps}\mu(\B) = \infty.
\end{align*}
The proof is completed with the next Lemma.
\end{proof}

\begin{lemma}\label{lem:2ndmoment}
If $\lambda_2 \in L^2(\S,\mu)$ then $\Ex{Y_1 Y_2} \leq (1+o(1)) \Ex{Y_1}\Ex{Y_2}$.
\end{lemma}
\begin{proof}
Define the ``good'' set $\G = \{x\in\S: \lambda_2(x) \leq \sqrt{n}/\log^2 n\}$ and let $G$ be the event that both $X_1\in \G$ and $X_2\in \G$. Then,
\[
\Ex{Y_1 Y_2} = \Ex{Y_1 Y_2 \ind{G^c}} + \Ex{Y_1 Y_2 \ind{G}}.
\]
For the first term, note that for $i\neq j$ we have 
\[
\Ex{Y_i f(X_j)} \leq \Ex{\ind{[X_i\in \B]}\prod_{\ell\neq i,j} \ind{[(i,\ell)\notin E_n]} f(X_j)} \leq \Ex{Y_i} \Ex{f(X_j)}.
\]
Therefore,
\begin{align*}
\Ex{Y_1 Y_2 \ind{G^c}}
&\leq \Ex{Y_1 \ind{[X_2\notin \G]}} + \Ex{Y_2 \ind{[X_1\notin \G]}} \\
&\leq \Ex{Y_1} \Pr{X_2\notin \G} + \Ex{Y_2} \Pr{X_1\notin \G}.
\end{align*}
Using Chebyshev's inequality
\[
\Pr{X\notin \G} = \Pr{\lambda_2(X) > \sqrt{n}/\log^2 n} \leq \frac{\norm{\lambda_2}_2^2 \log^4 n}{n} = o(n^{\eps-1}),
\]
since we have $\Ex{\lambda_2(X)^2} = \int_{\S} \lambda(x)^2 d\mu(x) = \norm{\lambda_2}_2^2 < \infty$. Thus, we have that
\[
\Ex{Y_1 Y_2 \ind{G^c}} \leq o(1)\Ex{Y_1}\Ex{Y_2}.
\]
Now for the second term, for $Y_1 Y_2 = 1$ no vertex $i > 2$ can be adjacent to $1$ or $2$ so
\begin{equation}
\Ex{Y_1 Y_2 \ind{G}} \leq \int_\G\int_\G \Ex{\big(1-\K(X,x)p_n\big)_+ \big(1-\K(X,y)p_n\big)_+}^{n-2} d\mu(x)d\mu(y). \label{eq:step1}
\end{equation}
We can bound the integrand by
\begin{align}
&\Ex{\big(1-\K(X,x)p_n\big)_+ \big(1-\K(X,y)p_n\big)_+} \nonumber \\
&\qquad\qquad\leq \Ex{\exp{-\big(\K(X,x)+\K(X,y)\big)p_n \Big.}} \nonumber \\
&\qquad\qquad\leq \Ex{1-\big(\K(X,x)+\K(X,y)\big)p_n + \frac{1}{2}\big(\K(X,x)+\K(X,y)\big)^2 p_n^2 } \nonumber \\
&\qquad\qquad= 1-\big(\lambda(x)+\lambda(y)\big)p_n + \frac{1}{2}\Ex{\big(\K(X,x)+\K(X,y)\big)^2} p_n^2. \label{eq:step2}
\end{align}
Since $\lambda_2(x) = \sqrt{\Ex{\K(X,x)^2}}$, by the Cauchy--Schwarz inequality, we have
\begin{align}
\Ex{\big(\K(X,x)+\K(X,y)\big)^2}
&= \Ex{\K(X,x)^2} + 2 \Ex{\K(X,x)\K(X,y)} + \Ex{\K(X,y)^2} \nonumber \\
&\leq \lambda_2(x)^2 + 2\lambda_2(x)\lambda_2(y) + \lambda_2(y)^2 \nonumber \\
&= \big( \lambda_2(x) + \lambda_2(y) \big)^2. \label{eq:step3}
\end{align}
Combining the bounds from equations \eqref{eq:step2} and \eqref{eq:step3} we obtain
\begin{align*}
&\Ex{\big(1-\K(X,x)p_n\big)_+ \big(1-\K(X,y)p_n\big)_+} \\
&\qquad\qquad\leq 1 - \big(\lambda(x)+\lambda(y)\big)p_n + \frac{1}{2}\big(\lambda_2(x)+\lambda_2(y)\big)^2p_n^2 \\
&\qquad\qquad= \Big(1 - \big(\lambda(x)+\lambda(y)\big)p_n\Big) \left(1+ \frac{1}{2}\cdot\frac{\big(\lambda_2(x)+\lambda_2(y)\big)^2p_n^2}{1 - \big(\lambda(x)+\lambda(y)\big)p_n}\right) \\
&\qquad\qquad\leq \Big(1 - \big(\lambda(x)+\lambda(y)\big)p_n\Big) \Big(1+ \big(\lambda_2(x)+\lambda_2(y)\big)^2p_n^2 \Big),
\end{align*}
for $n$ large enough since $\lambda(x) < 1$ for all $x\in\B$. Furthermore, if $x,y\in \G$ we have that
\begin{equation*}
\Ex{\big(1-\K(X,x)p_n\big)_+ \big(1-\K(X,y)p_n\big)_+} 
\leq \Big(1 - \big(\lambda(x)+\lambda(y)\big)p_n\Big) \left(1+ \frac{4n p_n^2}{\log^4 n} \right).
\end{equation*}
From this and the bound in equation \eqref{eq:step1} we get
\begin{align*}
\Ex{Y_1 Y_2 \ind{G}}
&\leq \int_\G\int_\G \Big(1 - \big(\lambda(x)+\lambda(y)\big)p_n\Big)^{n-2} \left(1+ \frac{4n p_n^2}{\log^4 n} \right)^{n-2} d\mu(x)d\mu(y) \\
&\leq \int_\G\int_\G \Big(1 - \big(\lambda(x)+\lambda(y)\big)p_n\Big)^{n-2} \exp{\frac{4n(n-2) p_n^2}{\log^4 n}} d\mu(x)d\mu(y) \\
&\leq (1+o(1))\int_\B\int_\B \Big(1 - \big(\lambda(x)+\lambda(y)\big)p_n\Big)^{n-2} d\mu(x)d\mu(y).
\end{align*}
Note that since the right term of the inequality in \eqref{eq:expected} is positive we have
\begin{align*}
\Ex{Y_1}\Ex{Y_2}
&\geq \int_\B (1-\lambda(x)p_n)^{n-1} d\mu(x) \int_\B (1-\lambda(y)p_n)^{n-1} d\mu(y) \\
&= \int_\B\int_\B (1-\lambda(x)p_n)(1-\lambda(y)p_n)^{n-1}  d\mu(x)d\mu(y) \\
&\geq \int_\B\int_\B \Big(1-\big(\lambda(x)+\lambda(y)\big)p_n\Big)^{n-1}  d\mu(x)d\mu(y). 
\end{align*}
Therefore, the proof is complete since we have
\[
\Ex{Y_1 Y_2 \ind{G}} \leq (1+o(1))\Ex{Y_1}\Ex{Y_2}. \qedhere
\]
\end{proof}

%%%%%%%%%%%%%%%%%%%%%%%%%%%%%%%%%%%%%%%%%%%%%%%%%%%%
\section{Connectivity threshold}\label{sec:connectivity}

The objective of this part is to prove that once the graph does not have isolated vertices, which happens when $\lam > 1$, then there is only one connected component, i.e., the graph is connected. The proof has two parts. First we prove that every component is of linear size, and then we show that any pair of linear-sized sets are connected. 

%%%%%%%%%%%%%%%%%%%%%%%%%%
\subsection{Every component is large}

To prove that there are no small components we use a first moment bound. Given two sets of vertices $A,B$ we write $A\notcx B$ for the event that $A$ does not connect to $B$, i.e., $A\notcx B = \cap_{i\in A}\cap_{j\in B} \{(i,j)\notin E_n\}$.

\begin{lemma}\label{lem:laplace}
Let $\lambda_2\in L^\infty(\S,\mu)$. Then for $1\leq k < n$ and any set $A\subset\{1,\dots,n\}$ of size $|A| = k$ we have
\[
\Pr{A\notcx A^c} \leq \left(1 - \lam k p_n + \norm{\lambda_2}_\infty^2 k^2 p_n^2/2\Big.\right)^{n-k}.
\]
\end{lemma}
\begin{proof}
Without loss of generality, assume $A=\{1,\dots,k\}$. We have
\begin{align*}
\Pr{A\notcx A^c} &= \Pr{\bigcap_{j\in A^c}\bigcap_{i\in A}\; (i,j)\notin E_n} \\
&= \Ex{\prod_{j\in A^c}\prod_{i\in A} \big(1-\K(X_j,X_i)p_n\big)_+} \\
&= \int_\S\cdots\int_\S \prod_{j\in A^c}\Ex{\prod_{i=1}^k \big(1-\K(X_j,x_i)p_n\big)_+} d\mu(x_1)\dots d\mu(x_k) \\
&\leq \int_\S\cdots\int_\S \left(1 - \sum_{i=1}^k\lambda(x_i) p_n + \frac{\norm{\lambda_2}_\infty^2  k^2 p_n^2}{2} \right)^{n-k} d\mu(x_1)\dots d\mu(x_k) \\
&\leq \left(1 - \lam k p_n + \norm{\lambda_2}_\infty^2 k^2 p_n^2 /2 \Big.\right)^{n-k},
\end{align*}
where the first inequality above follows from
\begin{align*}
\Ex{\prod_{i=1}^k \big(1-\K(X,x_i)p_n\big)_+}
&\leq \Ex{\exp{-\sum_{i=1}^k \K(X,x_i)p_n}} \\
&\leq \Ex{1 - \sum_{i=1}^k \K(X,x_i)p_n + \frac{1}{2}\left(\sum_{i=1}^k \K(X,x_i)\right)^2 p_n^2 } \\
&\leq 1 - \sum_{i=1}^k \lambda(x_i) p_n + \frac{\norm{\lambda_2}_\infty^2 k^2 p_n^2}{2},
\end{align*}
which holds because
\begin{align*}
\Ex{\left(\sum_{i=1}^k \K(X,x_i)\right)^2}
&= \sum_{i=1}^k\sum_{j=1}^k \Ex{\K(X,x_i)\K(X,x_j)} \\
&\leq \sum_{i=1}^k\sum_{j=1}^k \lambda_2(x_i)\lambda_2(x_j) \\
&\leq \norm{\lambda_2}_\infty^2 k^2. \qedhere
\end{align*}
\end{proof}

To get rid of larger components we need the following result which is based on the concentration of the number of edges of the graph.

\begin{lemma}\label{lem:concentration}
Let $\lambda_2\in L^\infty(\S,\mu)$. Then for $1\leq k \leq n/2$ and any set $A\subset\{1,\dots,n\}$ of size $|A| = k$ we have
\[
\Pr{A\notcx A^c} \leq e^{-p_n \lam k(n-k)/2} +  ke^{-n \lam^2 / 16 \norm{\lambda_2}_\infty^2}.
\]
\end{lemma}
\begin{proof}
Without loss of generality assume $A=\{1,\dots,k\}$. We have
\begin{align}
\Pr{A\notcx A^c} &= \Pr{\bigcap_{j\in A^c}\bigcap_{i\in A}\; (i,j)\notin E_n} \notag \\
&= \Ex{\prod_{i\in A}\prod_{j\in A^c} \big(1-\K(X_i,X_j)p_n\big)_+} \notag \\
&\leq \Ex{\exp{-p_n \sum_{i\in A}\sum_{j\in A^c} \K(X_i,X_j)}} \notag \\
&= \int_\S\cdots\int_\S \Ex{e^{-p_n \sum_{i=1}^k Z(x_i)}} d\mu(x_1)\dots d\mu(x_k), \label{eq:expsum}
\end{align}
where we define $Z(x_i) = \sum_{j\in A^c} \K(x_i,X_j)$.

We use the following Bernstein type inequality:
\textit{If $Y_1,Y_2,\dots,Y_n$ are non-negative independent random variables and $Y = \sum_{j=1}^n Y_j$ then}
\[
\Pr{Y\leq \E Y - t} \leq e^{-t^2 / 2 \sum_{j=1}^n \E Y_j^2}.
\]
See Theorem 3.5 of \cite{CL06a} (also the monograph \cite{McD98} or chapter 2 of the book \cite{CL06b}). \\
For every $1\leq i\leq k$ we apply the inequality to $Y = Z(x_i)$ with $Y_j = \K(x_i,X_j)$ so that $\E Y_j^2 = \lambda_2(x_i)^2$ and $t = \E Y / 2 = \E Z(x_i) / 2 = \lambda(x_i)(n-k) / 2$ to obtain
\[
\Pr{ Z(x_i) \leq \frac{\E Z(x_i)}{2} \Big.} \leq e^{-\lambda(x_i)^2 (n-k) / 8 \lambda_2(x_i)^2}.
\]
Let $\U = \{x\in \S: \lambda(x) \geq \lam \text{ and } \lambda_2(x) \leq \norm{\lambda_2}_{\infty} \}$. Note that $\mu(\S\setminus\U) = 0$. If $x_i \in \U$ for all $i=1,\dots,k$, we have
\begin{align*}
\Pr{Z(x_i) \leq \frac{\lam (n-k)}{2}}
&\leq \Pr{ Z(x_i) \leq \frac{\E Z(x_i)}{2} \Big.} \\
&\leq e^{-\lambda(x_i)^2 (n-k) / 8 \lambda_2(x_i)^2} \\
&\leq e^{- n \lam^2 / 16 \norm{\lambda_2}_\infty^2},
\end{align*}
since $k \leq n/2$. Using the union bound we get
\[
\Pr{ \sum_{i=1}^k Z(x_i) \leq \frac{\lam k(n-k)}{2} } \leq
\Pr{\min_{i\in A} Z(x_i) \leq \frac{\lam (n-k)}{2}} \leq
k e^{- n \lam^2 / 16 \norm{\lambda_2}_\infty^2}.
\]

Let $E = E(x_1,\dots,x_k)$ be the event where $\sum_{i=1}^k Z(x_i) \geq \lam k(n-k)/2$. Then, using inequality \eqref{eq:expsum} we can write
\begin{align*}
\Pr{A\notcx A^c}
&\leq \int_\S\cdots\int_\S \Ex{e^{-p_n \sum_{i=1}^k Z(x_i)}(\ind{E}+\ind{E^c})} d\mu(x_1)\dots d\mu(x_k) \\
&\leq \int_\S\cdots\int_\S \left( \Ex{e^{-p_n \sum_{i=1}^k Z(x_i)}\ind{E}}+ \Pr{E^c} \right) d\mu(x_1)\dots d\mu(x_k) \\
&\leq \int_\U\cdots\int_\U \left( e^{-p_n \lam k(n-k)/2} +  ke^{-n \lam^2 / 16 \norm{\lambda_2}_\infty^2} \right) d\mu(x_1)\dots d\mu(x_k) \\
&\leq e^{-p_n \lam k(n-k)/2} +  ke^{-n \lam^2 / 16 \norm{\lambda_2}_\infty^2}.\qedhere
\end{align*}
\end{proof}

\begin{proposition}\label{prop:compsize}
Let $\lambda_2\in L^\infty(\S,\mu)$ and $\lam > 1$. Then, there exists $\delta > 0$ such that all connected components of $G(n,\K)$ have size greater than $\delta n$ with high probability.
\end{proposition}
\begin{proof}
Let $N_k$ denote the number of components of size exactly $k$ and $A=\{1,\dots,k\}$. By Lemma \ref{lem:laplace}, we have that
\begin{align}
\E N_k &\leq \binom{n}{k} \Pr{A\notcx A^c} \notag \\
&\leq n^k \left(1 - \lam k p_n + \norm{\lambda_2}_\infty^2 k^2 p_n^2/2 \Big.\right)^{n-k} \notag \\
&\leq \exp{k\log n - (n-k)\lam k p_n + (n-k) \norm{\lambda_2}_\infty^2 k^2 p_n^2/2 \Big.} \notag \\
&\leq \exp{k\log n \left(1 - \lam + \frac{\lam k}{n} + \frac{\norm{\lambda_2}_\infty^2 (n-k) k \log n}{2n^2} \right) } \label{eq:terms} \\
&\leq e^{-(\lam-1) k\log n / 2}, \notag
\end{align}
for $k = o(n / \log n)$ because $k/n \to 0$ and $k \log n / n \to 0$, which implies that the last two terms in equation \eqref{eq:terms} are smaller than $\eps = (\lam -1)/4$ for $n$ large enough. Therefore,
\begin{align*}
\Pr{\sum_{k=1}^{en^{3/4}} N_k > 0}
&\leq \sum_{k=1}^{en^{3/4}} \E N_k 
&\leq \sum_{k=1}^{en^{3/4}}  e^{-(\lam-1) k\log n / 2} 
&\leq \frac{e^{-(\lam-1) \log n / 2}}{1-e^{-(\lam-1) \log n / 2}} \to 0.
\end{align*}
Fix $0< \delta \leq 1/2$ to be chosen later. For the rest of the range using Lemma \ref{lem:concentration} we obtain
\begin{align*}
\Ex{N_k} &\leq \binom{n}{k} \Pr{A\notcx A^c} \\
&\leq \left(\frac{ne}{k}\right)^k \left(e^{-p_n \lam k(n-k)/2} +  ke^{-n \lam^2 /16 \norm{\lambda_2}_\infty^2} \right) \\
&\leq \underbrace{ \left(\frac{ne}{k}\right)^k e^{-p_n \lam k(n-k)/2} }_{[LT]}
+ \underbrace{ \left(\frac{ne}{k}\right)^k ke^{-n \lam^2 /16 \norm{\lambda_2}_\infty^2} } _{[RT]}.
\end{align*}
For the left term we have
\[
[LT] \leq \exp{k\left(1 + \log n - \log k - \frac{\lam}{4} \log n \right)} \leq e^{-(\lam-1) k \log n / 4},
\]
if $k > en^{3/4}$. While for the right term
\[
[RT] \leq k \cdot \exp{n\left( \frac{k}{n}\:\: - \frac{k}{n} \log\frac{k}{n}\:\: -\frac{\lam^2}{16 \norm{\lambda_2}_\infty^2}\right)} \leq n e^{-n  \lam^2 / 32 \norm{\lambda_2}_\infty^2},
\]
if $k/n < \delta$ where $\delta = \max\left\{\rho\in[0,1/2]: \rho - \rho \log \rho \leq \lam^2 / 32 \norm{\lambda_2}_\infty^2 \right\} > 0$. Therefore,
\begin{align*}
\Pr{\sum_{k > en^{3/4}}^{\delta n} N_k > 0}
&\leq \sum_{k > en^{3/4}}^{\delta n} \E N_k \\
&\leq n \left(e^{-(\lam-1) n^{3/4} \log n / 2} + ne^{-n \lam^2 /32 \norm{\lambda_2}_\infty^2} \right) \to 0.
\end{align*}
Thus we have proved that with high probability the graph has no component of size smaller than $\delta n$.
\end{proof}

%%%%%%%%%%%%%%%%%%%%%%%%%%
\subsection{All vertices are connected}

To prove that every vertex is connected we discretize the space $\S$ using a finite partition and work with a lower approximation of the kernel $\K$. For this approximation to behave nicely we need $\K$ to be continuous almost everywhere.  For $\A\subset \S$ we write $\diam(\A)=\sup\{d(x,y): x, y \in \A\}$, where $d$ is the metric on $\S$.

\begin{lemma}[Lemma 7.1 from \cite{BJR07}]\label{lem:partitions}
Given $(\S,\mu)$ there exists a sequence of finite partitions $\P_m=\{\A_{m,1},\dots,\A_{m,M_m}\}$, $m>1$, of $\S$ such that
\begin{enumerate}[$(a)$]
\item each $\A_{m,i}$ is measurable and $\mu(\partial \A_{m,i})=0$;
\item for each $m$, $\P_{m+1}$ refines $\P_m$, i.e., each $\A_{m,i}$ is a union $\cup_{j\in J_{m,i}} \A_{m+1,j}$ for some set $J_{m,i}$;
\item let $i_m(x)$ be such that $x\in \A_{m,i_m(x)}$, then $\diam(\A_{m,i_m(x)}) \to 0$ as $m\to \infty$ for $\mu$ almost every $x\in \S$.
\end{enumerate}
\end{lemma}

\begin{definition}
Given a sequence of partitions $\P_m$ as above, we define the \emph{lower approximation kernels} by
\[
\K_m(x, y) = \inf\{ \K(x', y') : x' \in \A_{m,i_m(x)}, y' \in \A_{m,i_m(y)} \},
\]
and the \emph{partition graphs} $H_m = (V_m, E_m)$ where the vertex set is given by $V_m = \big\{ 1\leq i\leq M_m: \mu(\A_{m,i}) > 0 \big\}$ and $(i,j)$ is an edge if $\K_m > 0$ in $\A_{m,i}\times \A_{m,j}$.
\end{definition}
Note that if $\K$ is continuous almost everywhere it holds that $\K_m(x,y) \nearrow \K(x,y)$ as $m\to \infty$, for almost every $(x,y) \in \S^2$.

\begin{lemma}\label{lem:giant}
If $\K$ is irreducible and continuous $(\mu\otimes\mu)$-almost everywhere, then for any $\eps > 0$ there exists $m>1$ and a connected component $C_m$ in $H_m$ with $\mu(\S\setminus \cup_{i\in C_m} \A_{m,i}) < \eps$.
\end{lemma}
\begin{proof}
We first show that we can find $m_0$ such that there exists $(i_0,j_0)\in E_{m_0}$. Since $\K \neq 0$ and is continuous almost everywhere there exists $(x_0,y_0)$ and $\delta > 0$ such that $\mu(B(x_0,\delta)), \mu(B(y_0,\delta)) > 0$ and if $d(x,x_0), d(y,y_0) < \delta$ then $\K(x,y) > 0$. Pick $m_0$ so that $\diam(\A_{m_0,i_{m_0}(x)}) < \delta$ and $\diam(\A_{m_0,i_{m_0}(y)}) < \delta$ then we have that $(i_{m_0}(x), i_{m_0}(y)) \in E_{m_0}$.

For $m \geq m_0$, since $\K_m \geq \K_{m_0} > 0$ on $\A_{m_0,i_0}\times \A_{m_0,j_0}$, we have that all the vertices $i\in V_m$ such that $\A_{m,i} \subseteq \A_{m_0,i_0}$ are in the same connected component of $H_m$ which we denote by $C_m$.

Let $\B_m = \cup_{i\in C_m} \A_{m,i}$ and $\S_m = \cup_{i\in V_m} \A_{m,i}$. If $i\in C_m$ and $j\notin C_m$ then $\K_m = 0$ on $\A_{m,i}\times \A_{m,j}$ therefore $\K_m = 0$ on $\B_m \times (\S_m \setminus \B_m)$ and thus almost everywhere on $\B_m \times (\S \setminus \B_m)$. Now define $\B = \cup_{m=1}^\infty \B_m$. If $n\geq m$, then $\B_m \subseteq \B_n$ so $\K_n =0$ almost everywhere on  $\B_m\times(\S\setminus \B) \subseteq \B_n\times (\S\setminus \B_n)$. Letting $n\to \infty$, we have $\K = 0$ almost everywhere on $\B_m\times(\S\setminus \B)$.  Taking the union in $m$ yields $\K = 0$ almost everywhere on $\B\times (\S \setminus \B)$.

Since $\K$ is irreducible, it follows that $\mu(\B) = 0$ or $\mu(\S\setminus \B) = 0$. As $\B \supseteq \B_{m_0} \supseteq \A_{m_0,j_0}$, we have $\mu(\B) > 0$, so $\mu(\S\setminus \B) = 0$. To finish the proof note that $\B_m \nearrow \B$ so $\mu(\S\setminus \B_m) \to 0$ and we can choose $m$ so that $\mu(\S\setminus \B_m) < \eps$.
\end{proof}

\begin{lemma}\label{lem:density}
Let $N(\mathcal{A}) = \#\{ X_i \in \mathcal{A} \}$ be the number of points in $\A$. Given a finite partition $\P_m$ of $\S$ with high probability for every $i = 1,\dots,M_m$
\[
n \mu(\A_{m,i}) / 2 < N(\A_{m,i}) < 2n \mu(\A_{m,i}).
\]
\end{lemma}
\begin{proof}
We use the binomial Chernoff bound \cite{Che52, Hoe63, JLR00}: If $\xi\sim \Bin(n,p)$ and $t>0$ then
\[
\min\left( \Pr{\xi \leq tnp}, \Pr{\xi \geq tnp} \Big.\right) \leq e^{-f(t)np},
\]
where we write $f(x)=x\log x-x+1$. For a fixed set $\A_{m,i}$, the number of points $N(\A_{m,i})$ is $\Bin(n,\mu(\A_{m,i}))$. Thus, we have for any $1\leq i \leq M_m$,
\begin{align*}
\Pr{N(\A_{m,i}) \leq n\mu(\A_{m,i})/2\big.} &\leq e^{-f(1/2) n\mu(\A_{m,i})}, \\
\Pr{N(\A_{m,i}) \geq 2n\mu(\A_{m,i})\big.} &\leq e^{-f(2) n\mu(\A_{m,i})}.
\end{align*}
For sets $\A_{m,i}$ of zero measure the result holds almost surely. Let $\alpha = \min\{\mu(\A_{m,i}): i\in V_m\}$ and define the events
\[
D_i = \left\{\frac{1}{2} < \frac{N(\A_{m,i})}{n\mu(\A_{m,i})} < 2 \right\}.
\] 
Since $f(1/2) < f(2)$ we have for all $i=1,\dots,M_m$
\[
\Pr{D_i^c} \leq 2 e^{-f(1/2)\alpha n}.
\]
We can apply a union bound to obtain
\[
\Pr{\bigcup_{i=1}^{M_m} D_i^c} \leq \sum_{i=1}^{M_m} \Pr{D_i^c}
\leq \sum_{i=1}^{M_m} 2 e^{-f(1/2)\alpha n} \leq 2M_me^{-f(1/2)\alpha n} \to 0. \qedhere
\]
\end{proof}

\begin{theorem}
If $\K$ is irreducible, continuous $(\mu\otimes\mu)$-almost everywhere, $\lambda_2\in L^\infty(\S,\mu)$ and $\lam > 1$, then $G(n,\K)$ is connected with high probability.
\end{theorem}
\begin{proof}
Assume that the graph is disconnected. Let $A$ be a connected component, by Proposition \ref{prop:compsize} it has size at least $\delta n$ with high probability. Consider the sequence of partitions $\P_m$ given in Lemma \ref{lem:partitions} and the associated partition graph $H_m=(V_m,E_m)$. Let $\eps = \delta/4$ by Lemma \ref{lem:giant} there exists $m>1$ and a connected component $C_m$ in $H_m$ with $\mu(\S\setminus \cup_{i\in C_m} \A_{m,i}) < \eps$. Let us fix such $m$ in the following.

By Lemma \ref{lem:density} the event $D = \cap_{i=1}^{M_m} \{ 1/2 < N(\A_{m,i})/n \mu(\A_{m,i}) < 2 \}$ holds with high probability. On $D$, the number of points in $\S\setminus \cup_{i\in C_m} \A_{m,i}$ is less than $2\eps n = \delta n/2$. Therefore, at least $\delta n /2$ points of $A$ must lie in sets $\A_{m,i}$ for $i\in C_m$. We can argue in the same way for $A^c$. By the pigeonhole principle there is at least $u,v\in C_m$ such that the number of points of $A$ in $\A_{m,u}$ is at least $\delta n / 2 |C_m|$ and the number of points of $A^c$ in $\A_{m,v}$ is at least $\delta n / 2 |C_m|$.

Now define a function $f:C_m\to \{0,1\}$ in the following way: $f(u)=1, f(v)=0$, and for any other vertex $f(i) = 1$ if the majority of points in $\A_{m,i}$ belongs to $A$ and $f(i) = 0$ otherwise. Consider a path $u=i_0,i_1,\dots,i_\ell=v$ between $u$ and $v$ in $C_m$, such a path exists since $C_m$ is connected. Let $q = \min\{1\leq k\leq \ell : f(i_k) = 0\}$ then $f(i_{q-1}) = 1$ and $f(i_q) = 0$.

Let $\alpha = \min \{\mu(\A_{m,i}): i\in V_m\}$, clearly $\alpha > 0$ because $V_m$ is finite. Let $\beta_{i,j} = \inf\{\K(x,y): x\in\A_{m,i}, y\in\A_{m,j} \}$ note that $\beta_{i,j} > 0$ for any edge $(i,j)\in E_m$ of the partition graph $H_m$. Define $\beta = \min\{\beta_{i,j}: (i,j)\in E_m\}$, thus $\beta > 0$ since $E_m$ is finite. Define $U = \{i\in A: X_i\in\A_{m,i_{q-1}} \}$ and $V = \{i\in A^c: X_i\in\A_{m,i_q} \}$. On $D$, we have that $|U|,|V| \geq \gamma n$ where $\gamma = \min\{\alpha/2, \delta / 2 |C_m| \}$. Therefore conditionally on $D$ we have
\begin{align*}
\Pr{A\notcx A^c \mid D} &\leq \Pr{U\notcx V \mid D} \\
&\leq \Ex{\prod_{i\in U}\prod_{j\in V} \big(1-\K(X_i,X_j)p_n\big)_+ \;\bigg\vert\; D} \\
&\leq \Ex{(1-\beta p_n)^{|U||V|} \mid D} \\
&\leq (1-\beta p_n)^{\gamma^2 n^2} \\
&\leq e^{-\beta\gamma^2 n\log n}.
\end{align*}
We can apply this bound to finish the proof. As before, let $N_k$ be the number of components of size $k$. We have
\begin{align*}
\Pr{\sum_{k=\delta n}^{n/2} N_k > 0}
&\leq \Pr{D^c} + \Pr{\sum_{k=\delta n}^{n/2} N_k > 0 \;\bigg\vert\; D} \\
&\leq \Pr{D^c} + \sum_{k=\delta n}^{n/2} \Ex{N_k \mid D} \\
&\leq \Pr{D^c} + \sum_{k=\delta n}^{n/2} \binom{n}{k} \Pr{A\notcx A^c \mid D} \\
&\leq \Pr{D^c} + 2^n \times e^{-\beta\gamma^2 n\log n} \to 0.
\end{align*}
We have proved that with high probability there are no components of any size less than $n/2$. Thus, the graph is connected.
\end{proof}

%%%%%%%%%%%%%%%%%%%%%%%%%%%%%%%%%%%%%%%%%%%%%%%%%%%%
\section{Discussion}\label{sec:discuss}

When $\lam = 1$ we are in the window of connectivity. In this case the probability that the graph $G(n,\K)$ is connected doesn't go to either $0$ or $1$. For example, if $\K = 1$ then $G(n,\K)$ is just the random graph $G(n,p)$ with $p=\log n/n$. Erd\H{o}s and Renyi \cite{ER59} proved in this case that $\Pr{G(n,\K) \text{ is connected} } \to 1/e$ by showing that isolated vertices are still the main obstruction to obtain connectivity, i.e., with high probability the graph consists solely of a giant component and some isolated vertices and the number of them is asymptotically Poisson distributed.

The following example helps to illustrate that some integrability condition on $\lambda_2$ is necessary to obtain connectivity with high probability.
Let $\S = [0,1]$ and $\mu = m$ be the Lebesgue measure. Consider the following kernel
\[
\K(x,y) = \frac{c}{x}\; \ind{[x/2,x]}(y) + \frac{c}{y}\; \ind{[y/2,y]}(x).
\]
We have that $\lam = c/2$ because
\[
\lambda(x) = \begin{cases}
\frac{c}{2} + c \log 2 & \text{if } x \leq \frac{1}{2}, \\
\frac{c}{2} + c \log \frac{1}{x} & \text{if } x > \frac{1}{2}. \\
\end{cases}
\]
However, the graph $G(n,\K)$ is not connected with positive probability. To see this, consider the disjoint events $E_k = \{X_k < 1/n\} \cap \bigcap_{i\neq k} \{X_i > 2/n\}$. If $E_k$ holds then vertex $k$ is isolated in $G(n,\K)$.
Therefore,
\[
\Pr{\bigcup_{k=1}^n E_k} = \sum_{k=1}^n \Pr{E_k} = \sum_{k=1}^n \frac{1}{n}\left(1-\frac{2}{n}\right)^{n-1} 
= \left(1-\frac{2}{n}\right)^{n-1} \to \frac{1}{e^2}.
\]
This does not contradict Theorem \ref{thm:main} because this kernel has 
\[
\lambda_2(x) = \begin{cases}
\frac{c^2}{x} & \text{if } x \leq \frac{1}{2}, \\
\frac{3c^2}{2x} - c^2 & \text{if } x > \frac{1}{2}, \\
\end{cases}
\]
and thus $\lambda_2 \notin L^1(\S,\mu)$.

%%%%%%%%%%%%%%%%%%%%%%%%%%%%%%%%%%%%%%%%%%%%%%%%%%%%
\bibliographystyle{amsplain}
\bibliography{inhomogeneous}

\end{document}